\documentclass[11pt]{article}
\usepackage[a4paper,margin=3cm]{geometry}
\usepackage{amsmath,amssymb,amsthm,mathtools}
\usepackage{enumitem}
\usepackage{hyperref}
\hypersetup{colorlinks=true,linkcolor=blue,urlcolor=blue,citecolor=blue}

\newtheorem{theorem}{Theorem}[section]
\newtheorem{lemma}[theorem]{Lemma}
\newtheorem{proposition}[theorem]{Proposition}
\newtheorem{corollary}[theorem]{Corollary}

\DeclareMathOperator{\spanlin}{span}
\DeclareMathOperator{\UCP}{UCP}
\DeclareMathOperator{\Ext}{Ext}
\newcommand{\K}{\mathcal K}
\newcommand{\B}{\mathcal B}
\newcommand{\Q}{\mathcal Q}
\newcommand{\HH}{\mathcal H}
\newcommand{\LL}{\mathcal L}
\newcommand{\F}{\mathbb F}
\newcommand{\C}{\mathbb C}
\newcommand{\N}{\mathbb N}
\newcommand{\id}{\operatorname{id}}

\title{A Three-Dimensional Operator System without the Smith--Ward Property}
\author{Marcel Scherer}

\date{\today}

\begin{document}
\maketitle

\begin{abstract}
Harris recently showed that a non-liftable injective representation into the Calkin algebra gives explicit
four-dimensional operator systems in the Calkin algebra without the lifting
property, and hence a counterexamples to the generalized Smith--Ward problem for four-dimensional operator systems. The main obstruction also appears in an earlier work by Paulsen on this problem. We isolate the relevant part of this argument and replace the four-dimensional
operator system by a three-dimensional hyperrigid operator system inside a matrix
amplification of
\[
        C_r^*(\F_2).
\]
The resulting operator system in the Calkin algebra is of the form
\(\spanlin\{1,q(D),q(K)\}\), where $D$ and $K$ are selfadjoint operators, and the identity map on this operator system has no unital completely
positive lift. Equivalently, the operator $D+iK$ gives a counterexample to
the Smith--Ward problem.\\ 
By a result of Kavruk, the dual of this operator system fails to be exact, and hence is the first example of a three-dimensional operator system that is not exact.
\end{abstract}

\let\thefootnote\relax\footnotetext{
2020 \textit{Mathematics Subject Classification.} 46L07.\\
\ \textit{Key words and phrases.} Smith-Ward problem, lifting property, hyperrigidity, $C^*$-nuclear.\\
 The author was supported by a Technion Fellowship.}

\section{Introduction}

Let \(\HH\) be a separable infinite-dimensional Hilbert space and let
\[
        q:\B(\HH)\longrightarrow \Q(\HH):=\B(\HH)/\K(\HH)
\]
be the quotient map onto the Calkin algebra.  The Smith--Ward theorem says that,
for an operator \(T\in\B(\HH)\) and a fixed \(N\in\N\), the first \(N\)
essential matrix ranges of \(T\) can be realized exactly as the first \(N\)
matrix ranges of a compact perturbation of \(T\) \cite{SmithWard}.  The
classical Smith--Ward problem asks whether this compact perturbation can be
chosen independently of \(N\).\\
As shown in \cite{Paul82}, the Smith--Ward
problem is equivalent to the lifting problem for
three-dimensional operator systems of the form
\[
        \spanlin\{q(I),q(T),q(T)^*\}\subseteq \Q(\HH),
\]
that is, whether the identity map on this operator system has a u.c.p. lift to $\mathcal{B}(\HH)$. More generally, for an operator system $S\subset\mathcal{B}(\HH)$, the generalized Smith--Ward problem asks whether the identity map on $q(S)$ has a u.c.p. lift to $\mathcal{B}(\HH)$.\\
In \cite{Paul82}, it was shown that there exists a $5$-dimensional operator system in $\Q(\HH)$ without the lifting property, and this was improved to a $4$-dimensional operator system in \cite{Kavruk}.
Harris gives an explicit
four-dimensional example: from a non-liftable injective representation into the Calkin algebra, one obtains a
four-dimensional operator subsystem of a matrix amplification whose identity map
has no u.c.p. lift, see \cite[Theorem 2 and
Corollary 9]{Harris}.  His construction, as well as the construction in \cite[Theorem 3.3]{Paul82}, is based on a multiplicative-domain argument. \\
The purpose of this note is to replace Harris' four-dimensional operator system by
a three-dimensional one.  The new point is to construct a hyperrigid operator system
\[
        S=\spanlin\{1,D,K\}\subseteq M_4(C_r^*(\F_2)),
\]
where \(D\) is a diagonal selfadjoint matrix with four distinct eigenvalues and
\(K\) is a selfadjoint matrix whose non-zero entries are labelled by the two
canonical unitaries of \(C^*_r(\mathbb{F}_2)\) and by the unit.\\
Although Arveson's original definition of hyperrigidity in \cite{Arve1} is phrased in
terms of approximation by u.c.p. maps, we use the equivalent
unique-extension formulation: for every unital representation \(\pi\) of
\(C^*(S)\), the restriction \(\pi|_S\) has a unique u.c.p. extension to
\(C^*(S)\), namely \(\pi\) itself. In this situation, one
says that \(\pi|_S\) has the \textit{unique extension property}.
The proof that $S$ is hyperrigid starts by taking a u.c.p. extension of $\pi|_S$ to $C^*(S)$ and dilating the extension via Stinespring to a unital representation $\rho$ of $C^*(S)$. Then we will show that the dilation $\rho|_S$ of $\pi|_S$ is already \textit{trivial}, meaning that if $\pi:S\to\mathcal{B}(\HH)$, $\rho:S\to B(\mathcal L)$ and $V:H\to \mathcal L$ is an isometry with 
  \[
    \pi(s)=V^*\rho(s) V, \qquad s\in S
  \]
then $V\pi(s)=\rho(s)V$ for all $s\in S$. See \cite{Arve2} for more  details on the relation between unique extensions and trivial dilations.\\
Combining this three-dimensional hyperrigid operator system with Harris' obstruction to the Smith--Ward property gives the following result.

\vskip 6 pt
\noindent {\bf Theorem \ref{thm: main}.} \textit{There exists an operator \(T\in\B(\HH)\) such that no compact operator \(L\in\K(\HH)\) satisfies
\[
        \mathcal W(T+L)
        =
        \mathcal W(q(T)),
\]
where \(\mathcal W\) denotes the full joint matrix range.}
\vskip 6 pt

\section{Preliminaries}

For a tuple \(a=(a_1,\ldots,a_m)\) in a unital \(C^*\)-algebra, write
\[
        \mathcal W_n(a)
        =
        \{(\phi(a_1),\ldots,\phi(a_m)):\phi\in\UCP(C^*(a),M_n)\}
\]
for its \(n\)-th joint matrix range and
\(\mathcal W(a)=\bigsqcup_{n\geq 1}\mathcal W_n(a)\) for the full joint matrix
range.  We use the standard matrix-range duality, see for example \cite[Theorem 5.1]{DDSS17}: for operator systems generated by tuples \(a\)
and \(b\), inclusion of full matrix ranges is equivalent to the existence of a
unital completely positive map sending the generators \(a_i\) to \(b_i\).  In
particular, if
\[
        \mathcal W(a)=\mathcal W(b),
\]
then the unital linear map 
  \[
    \textup{span}\{1,a_1,\dots,a_m\}\to\textup{span}\{1,b_1,\dots,b_m\}, a_i\mapsto b_i
  \]
is completely positive and bijective with completely positive inverse.\\
We also use the following standard consequence of Stinespring dilation theory.
It is the Ext-theoretic obstruction underlying Harris' and Paulsen's constructions; we include
the proof to fix the direction of the argument.

\begin{lemma}
\label{lem:ucp-lift-invertible}
Let \(E\) be a unital simple separable \(C^*\)-algebra and let
\(\tau:E\to\Q(\HH)\) be a unital representation, where $\HH$ is separable.  If \(\tau\) has a
u.c.p. lift \(\sigma:E\to\B(\HH)\), i.e. \(q\circ\sigma=\tau\), then the class
\([\tau]\in\Ext(E)\) is invertible.
\end{lemma}

\begin{proof}
Let
\[
        \sigma(e)=V^*\pi(e)V,\qquad e\in E,
\]
be a Stinespring dilation on a Hilbert space \(\LL\).  We identify \(\HH\) with
\(V\HH\subseteq\LL\), and let \(P\) be the projection onto this subspace.  By adding a separable infinite-dimensional summand to $\pi$, we can assume without loss of generality that $I-P$ is not compact.\\
Since
\(q\circ\sigma=\tau\) is multiplicative, we have for every
\(e\in E\),
\[
        \sigma(e^*e)-\sigma(e)^*\sigma(e)\in\K(\HH),
        \qquad
        \sigma(ee^*)-\sigma(e)\sigma(e)^*\in\K(\HH).
\]
In the Stinespring representation these defects are
\[
        P\pi(e)^*(1-P)\pi(e)P
        \quad\text{and}\quad
        P\pi(e)(1-P)\pi(e)^*P,
\]
so the off-diagonal corners
\((1-P)\pi(e)P\) and \(P\pi(e)(1-P)\) are compact. \\
Since $I-P$ is not compact, $\Q((1-P)\mathcal{L})\neq\{0\}$ and hence we may define a unital representation
\[
        \rho:E\to \Q((1-P)\mathcal{L}), e\mapsto q_{(1-P)\LL}((1-P)\pi(e)(1-P)).
\]
Since $E$ is simple and $\rho$ unital, $\rho$ is injective. Using the identification of $\mathcal{H}\cong P\mathcal{L}$, we have that
\[
   q_\LL\circ\pi=\tau\oplus\rho
  \]
lifts to the representation \(\pi:E\to B(\mathcal L)\).  Therefore \([\rho]\) is an inverse for \([\tau]\).
\end{proof}

Consider
\[
        C_r^*(\F_2).
\]
Write \(u,v\) for the canonical generators. By Powers' theorem, \(C^*_r(\mathbb{F}_2)\) is a simple $C^*$-algebra \cite{Powers}. Thus every unital representation, and in particular every unital representation into $\mathcal{Q}(\HH)$, is injective.

By a result of Haagerup--Thorbjornsen, we know that
\(\Ext(C^*_r(\F_2))\) is not a group \cite{HaagerupThorbjornsen}.  Thus there exists an injective representation
\[
        \tau:C^*_r(\F_2)\longrightarrow \Q(\HH)
\]
whose class \([\tau]\in\Ext(C^*_r(\F_2))\) is not invertible. \\
Set
\[
        \tau_4:=\id_{M_4}\otimes\tau:M_4(C^*_r(\F_2))\longrightarrow M_4(\Q(\HH))
        \cong \Q(\HH^{\oplus4}).
\]
The amplification $[\tau_4]$
represents the image of \([\tau]\) under the canonical isomorphism
\[
\operatorname{Ext}(C^*_r(\mathbb{F}_2))\cong \operatorname{Ext}(M_4(C^*_r(\mathbb{F}_2))).
\]
In particular, \([\tau_4]\) is non-invertible whenever \([\tau]\) is
non-invertible.

\section{The three-dimensional hyperrigid operator system}

Let \(e_{ij}\) be the standard matrix units of \(M_4(\C)\).  Define
\[
        D=\operatorname{diag}(1,2,3,4)\otimes1_{C^*_r(\mathbb{F}_2)}\in M_4(C^*_r(\mathbb{F}_2))
\]
and
\[
        K=
        \begin{pmatrix}
        0    & u & v & 1\\
        u^*  & 0 & 1 & 1\\
        v^*  & 1 & 0  & 0\\
        1    & 1 & 0 & 0
        \end{pmatrix}
        \in M_4(C^*_r(\mathbb{F}_2))
\]
Thus \(D=D^*\) and \(K=K^*\).  Put
\[
        S=\spanlin\{1,D,K\}\subseteq M_4(C^*_r(\F_2)).
\]

\begin{lemma}
\label{lem:generation}
It holds that
\[
        C^*(S)=M_4(C^*_r(\F_2)).
\]

\end{lemma}

\begin{proof}
Since \(D\) has the four distinct spectral values \(1,\ldots,4\), its spectral
projections
\[
        p_j:=e_{jj}\otimes1_{C^*_r(\F_2)},
        \qquad 1\leq j\leq4,
\]
belong to \(C^*(D)\subseteq C^*(S)\).  Therefore each corner \(p_iKp_j\)
belongs to \(C^*(S)\).  The entries of \(K\) equal to \(1_{C^*_r(\F_2)}\) give
  \[
        e_{14}\otimes1_{C^*_r(\F_2)}, e_{24}\otimes1_{C^*_r(\F_2)}, e_{23}\otimes1_{C^*_r(\F_2)}\in C^*(S).
\]
Simple computations now imply
  \[
    e_{ij}\otimes1_{C^*_r(\F_2)}\in C^*(S)        \qquad 1\leq i,j\leq4.
  \]
The labelled edges from the first row give
\[
        e_{12}\otimes u,
        \quad e_{13}\otimes v
        \in C^*(S).
\]
Multiplying by the already obtained scalar matrix units yields
\[
        e_{11}\otimes u,
        \quad e_{11}\otimes v
        \in C^*(S).
\]
Since \(u,v\) generate \(C^*_r(\F_2)\), and since all scalar matrix units belong to
\(C^*(S)\), we get \(M_4(C^*_r(\F_2))\subseteq C^*(S)\).  The reverse inclusion is
immediate.
\end{proof}

The next theorem is the point at which the present construction differs from
Harris' four-dimensional coding lemma \cite[Lemma 1]{Harris}.  Harris includes
a projection as one of the four generators.  Here the projection data are
recovered from the two endpoint spectral subspaces of \(D\) and propagated
with the help of the structure of \(K\). The argument essentially only uses that the diagonal operator $D$ has distinct eigenvalues, that the non-zero entries of $K$ are partial isometries, and the graph determined by these entries has a specific form. Thus the theorem holds in a more general setting. However, to keep the proof transparent, we restrict to the concrete $D$ and $K$ above.

\begin{theorem}
\label{thm:hyperrigid}
For every unital representation \(\pi:M_4(C^*_r(\F_2))\to\B(\mathcal E)\), the restriction \(\pi|_S\) has the unique extension
property: if \(\Phi:M_4(C^*_r(\mathbb{F}_2))\to\B(\mathcal E)\) is u.c.p. and
\(\Phi|_S=\pi|_S\), then \(\Phi=\pi\).
\end{theorem}

\begin{proof}
Let
\[
        \Phi(b)=V^*\rho(b)V,
        \qquad b\in M_4(C^*_r(\mathbb{F}_2)),
\]
be a Stinespring dilation, where \(\rho:M_4(C^*_r(\mathbb{F}_2))\to\B(\LL)\) is a unital
representation and \(V:\mathcal E\to\LL\) is an isometry.  For
\(1\leq j\leq4\), put
\[
        p_j=e_{jj}\otimes1_{C^*_r(\mathbb{F}_2)}.
\]
Then
\[
        \pi(D)=\sum_{j=1}^4 j\pi(p_j),
        \qquad
        \rho(D)=\sum_{j=1}^4 j\rho(p_j),
\]
and the operators $\pi(p_j)$ and $\rho(p_j)$ are orthogonal selfadjoint projections, possibly zero.\\
Since \(\Phi(D)=\pi(D)\), if \(\xi=\pi(p_1)x\in \pi(p_1)\mathcal E\), then
\[
        0=\left\langle\left(\sum_{j=2}^4(j-1)\pi(p_j)\right)\pi(p_1)x,\xi\right\rangle=\langle(\pi(D)-I)\xi,\xi\rangle
         =\langle(\rho(D)-I)V\xi,V\xi\rangle.
\]
As \(\rho(D)-I=\sum_{j=2}^4(j-1)\rho(p_j)\geq0\), it follows that
  \[
    P_{V(\pi(p_1)\mathcal{E})}(\rho(D)-I)P_{V(\pi(p_1)\mathcal{E})}=0
 \]
and, since $\ker(\rho(D)-I)=\textup{Im}(\rho(p_1))$, we get
  \[
    V(\pi(p_1)\mathcal{E})\subset\ker(\rho(D)-I)=\rho(p_1)\LL.
  \]
Similarly, since $\Phi(D)=\pi(D)$, we have for $\xi=\pi(p_4)x\in\pi(p_4)\mathcal{E}$ that
\[
        0=\left\langle\left(\sum_{j=1}^3(4-j)\pi(p_j)\right)\pi(p_4)x,\xi\right\rangle=\langle(4I-\pi(D))\xi,\xi\rangle
         =\langle(4I-\rho(D))V\xi,V\xi\rangle.
\]
 Using \(4I-\rho(D)=\sum_{j=1}^3(4-j)\rho(p_j)\geq0\), one obtains
  \[
    P_{V(\pi(p_4)\mathcal{E})}(4I-\rho(D))P_{V(\pi(p_4)\mathcal{E})}=0
 \]
 and, since $\ker(4I-\rho(D))=\textup{Im}(\rho(p_4))$, we get
\[
        V(\pi(p_4)\mathcal E)\subseteq \rho(p_4)\LL.
\]
Write
\[
        t_{ij}:=p_iKp_j=e_{ij}\otimes a_{ij}.
\]
Whenever $t_{ij}\neq0$, the element \(a_{ij}\in C^*_r(\mathbb{F}_2)\) is a unitary, and hence \(t_{ij}\) is a partial
isometry with  \(p_j=t_{ij}^*t_{ij}\) and \(p_i=t_{ij}t_{ij}^*\).\\
Fix for a moment a $j\in\{1,2,3,4\}$ and assume that \(V\pi(p_j)\mathcal{E}\subseteq\rho(p_j)\LL\) (we only know this for $j=1,4$ thus far). Then for \(\eta\in \pi(p_j)\mathcal E\), the orthogonality of the \(\rho(p_i)\)'s and the assumption gives
\[
        \|\rho(K)V\eta\|^2
        =
        \left\|\sum_{i=1}^4\rho(t_{ij})V\eta\right\|^2
        =
        \sum_{i=1}^4\left\|\rho(t_{ij})V\eta\right\|^2
        =
        \begin{cases} 2\|\eta\|^2 & j=3,4\\ 3\|\eta\|^2 & j=1,2. \end{cases}
\]
Likewise, we can deduce from the orthogonality of the $\pi(p_i)$'s that
\[
        \|\pi(K)\eta\|^2=\begin{cases} 2\|\eta\|^2 & j=3,4\\ 3\|\eta\|^2 & j=1,2. \end{cases}
\]
Since \(\Phi(K)=\pi(K)\),
\[
        V^*\rho(K)V\eta=\pi(K)\eta.
\]
The contraction \(V^*\) therefore preserves the norm of \(\rho(K)V\eta\), so
\(\rho(K)V\eta\in V\mathcal E\), and hence
\[
        \rho(K)V\eta=V\pi(K)\eta.
\]
Now we show that
\[
        V(\pi(p_2)\mathcal E)\subseteq \rho(p_2)\LL.
\]
By the above norm argument, we have for $\eta\in \pi(p_4)\mathcal{E}$ that
\[
        \rho(K)V\eta=V\pi(K)\eta.
\]
Projecting this equality onto \(\rho(p_2)\LL\) yields
  \[
    \rho(p_2)\rho(K)V\eta=\rho(p_2)V\pi(K)\eta=\rho(p_2)V\pi(K)\pi(p_4)\eta=\rho(p_2)V\pi(t_{14}+t_{24})\eta.
  \]
Using
\[
        V(\pi(p_4)\mathcal E)\subseteq \rho(p_4)\LL 
\]
the left hand side changes to  
  \[ 
    \rho(p_2)\rho(K)V\eta=\rho(t_{21}+t_{23}+t_{24})V\eta=\rho(t_{24})V\eta, 
  \]
while using
  \[
    V(\pi(p_1)\mathcal{E})\subseteq\rho(p_1)\LL
  \]
the right hand side changes to
  \[
    \rho(p_2)V\pi(t_{14}+t_{24})\eta=\rho(p_2)V\pi(t_{24})\eta.
  \]
So we obtain
\[
        \rho(t_{24})V\eta=\rho(p_2)V\pi(t_{24})\eta.
\]
The left hand side has norm \(\|\eta\|\), while
\(V\pi(t_{24})\eta\) also has norm \(\|\eta\|\).  Thus
\[
        V\pi(t_{24})\eta\in \rho(p_2)\LL.
\]
As \(\pi(t_{24})\) maps \(\pi(p_4)\mathcal E\) onto \(\pi(p_2)\mathcal E\), the claim
follows.\\
Analogously we obtain
\[
        V(\pi(p_3)\mathcal E)\subseteq \rho(p_3)\LL.
\]
Indeed, take \(\eta\in \pi(p_2)\mathcal E\). The norm argument implies that
\[
        \rho(K)V\eta=V\pi(K)\eta.
\]
Projecting this equality onto \(\rho(p_3)\LL\) yields
  \[
    \rho(p_3)\rho(K)V\eta=\rho(p_3)V\pi(K)\eta=\rho(p_3)V\pi(K)\pi(p_2)\eta=\rho(p_3)V\pi(t_{12}+t_{32}+t_{42})\eta.
  \]
Using
\[
        V(\pi(p_2)\mathcal E)\subseteq \rho(p_2)\LL
\]
the left hand side changes to  
  \[ 
    \rho(p_3)\rho(K)V\eta=\rho(t_{31}+t_{32})V\eta=\rho(t_{32})V\eta, 
  \]
while using
  \[
    V(\pi(p_1)\mathcal{E})\subseteq\rho(p_1)\LL,\qquad V(\pi(p_4)\mathcal E)\subseteq \rho(p_4)\LL
  \]
the right hand side changes to
  \[
    \rho(p_3)V\pi(t_{12}+t_{32}+t_{42})\eta=\rho(p_3)V\pi(t_{32})\eta.
  \]
So we obtain
\[
        \rho(t_{32})V\eta=\rho(p_3)V\pi(t_{32})\eta.
\]
The left hand side has norm \(\|\eta\|\), while
\(V\pi(t_{32})\eta\) also has norm \(\|\eta\|\).  Thus
\[
        V\pi(t_{32})\eta\in \rho(p_3)\LL.
\]
As \(\pi(t_{32})\) maps \(\pi(p_2)\mathcal E\) onto \(\pi(p_3)\mathcal E\), the claim
follows.\\
Hence $V(\pi(p_i)\mathcal E)\subseteq \rho(p_i)\LL$ for $i=1,2,3,4$. Since $I=\sum_{i=1}^4\pi(p_i)$, it follows that 
  \[
    \begin{split}
\rho(D)V
&=
VV^*\rho(D)V+(I-VV^*)\rho(D)V\\
&=
V\pi(D)+\sum_{j=1}^4(I-VV^*)\rho(D)V\pi(p_j)\\
&=
V\pi(D)+\sum_{j=1}^4(I-VV^*)jV\pi(p_j)\\
&=
V\pi(D).  
  \end{split}
  \]
Moreover we now know that the norm argument holds for $j=1,2,3,4$ and hence
\[
        \rho(K)V\eta=V\pi(K)\eta
\]
for each $\eta\in\pi(p_j)\mathcal{E}$ and $j=1,2,3,4$. Using $I=\sum_{j=1}^4\pi(p_j)$ again, we get 
\[
        \rho(K)V\eta=V\pi(K)\eta\qquad \forall\eta\in\mathcal{E}.
\]
Thus \(V\mathcal E\) is invariant for \(\rho(D)\) and \(\rho(K)\).  Since both
\(D\) and \(K\) are selfadjoint, \(V\mathcal E\) is reducing for
\(\rho(D)\) and \(\rho(K)\).  By Lemma~\ref{lem:generation},
\(C^*(D,K)=M_4(C^*_r(\mathbb{F}_2))\), so \(V\mathcal E\) is reducing for \(\rho(M_4(C^*_r(\mathbb{F}_2)))\).  Therefore
\(\Phi\) is a \(*\)-homomorphism.  Since it agrees with \(\pi\) on the
generators \(D,K\), it agrees with \(\pi\) on all of \(M_4(C^*_r(\mathbb{F}_2))\).
\end{proof}

\section{Failure of the Smith--Ward property}

Let
\[
        \dot S:=\tau_4(S)\subseteq \Q(\HH^{\oplus4}).
\]
Since \(\tau_4\) is injective, \(\dot S\) is a three-dimensional operator system.

\begin{proposition}
\label{prop:no-lift}
The identity map
\[
        \id_{\dot S}:\dot S\longrightarrow\Q(\HH^{\oplus4})
\]
has no u.c.p. lift to \(\B(\HH^{\oplus4})\).
\end{proposition}

\begin{proof}
Assume that \(\sigma:\dot S\to\B(\HH^{\oplus4})\) is u.c.p. and satisfies
\[
        q\circ\sigma=\id_{\dot S}.
\]
Identifying \(S\) with \(\dot S\) via \(\tau_4|_S\), we get a u.c.p. map
\[
        \sigma_0:S\longrightarrow\B(\HH^{\oplus4})
\]
such that
\[
        q\circ\sigma_0=\tau_4|_S.
\]
By Arveson's extension theorem, \(\sigma_0\) extends to a u.c.p. map
\[
        \widetilde\sigma:M_4(C^*_r(\mathbb{F}_2))\longrightarrow\B(\HH^{\oplus4}).
\]
Then \(q\circ\widetilde\sigma:M_4(C^*_r(\mathbb{F}_2))\to\Q(\HH^{\oplus4})\) is a u.c.p. extension of
\(\tau_4|_S\).  Composing with a faithful representation of the Calkin algebra
and using Theorem~\ref{thm:hyperrigid}, we obtain
\[
        q\circ\widetilde\sigma=\tau_4
\]
on all of \(M_4(C^*_r(\mathbb{F}_2))\).  Thus \(\tau_4\) has a u.c.p. lift.  Lemma~\ref{lem:ucp-lift-invertible}
then implies that \([\tau_4]\) is invertible in \(\Ext(M_4(C^*_r(\mathbb{F}_2)))\), contradicting the
choice of \(\tau\).  Hence no such \(\sigma\) exists.
\end{proof}
\ \\

Now choose arbitrary selfadjoint lifts \(T_1,T_2\in\B(\HH^{\oplus4})\) of
\(\tau_4(D)\) and \(\tau_4(K)\), respectively:
\[
        q(T_1)=\tau_4(D),
        \qquad
        q(T_2)=\tau_4(K).
\]

\begin{theorem}
\label{thm: main}
For \(T=T_1+iT_2\), there is no compact operator
\(L\in\K(\HH^{\oplus4})\) such that
\[
        \mathcal W(T+L)
        =
        \mathcal W(q(T)).
\]
\end{theorem}

\begin{proof}
Suppose that such a compact perturbation exists and let $L_1=\textup{Re}(L), L_2=\textup{Im}(L)$.  Since
\[
        q(T_1+L_1)=q(T_1)=\tau_4(D),
        \qquad
        q(T_2+L_2)=q(T_2)=\tau_4(K),
\]
the equality of full matrix ranges implies that the unital map
\[
        \dot S\longrightarrow\B(\HH^{\oplus4})
\]
defined by
\[
        1\mapsto I,
        \qquad
        \tau_4(D)\mapsto T_1+L_1,
        \qquad
        \tau_4(K)\mapsto T_2+L_2
\]
is completely positive, see \cite{Kavruk}. Since this is a u.c.p. lift of \(\id_{\dot S}\), this contradicts
Proposition~\ref{prop:no-lift}.
\end{proof}

Combining this with \cite[Theorem 6.6]{Kavruk}, we obtain:

\begin{corollary}
The operator system dual $(\dot S)^d$ is a three-dimensional operator system that is not exact.
\end{corollary}

\section*{Acknowledgement}
The author thanks Orr Shalit for careful proofreading, which led to an improvement in the readability of the paper.
\ \\
\ \\

\textbf{AI Disclosure:} ChatGPT was used to perform literature search and to accelerate the search for the operator system $S$.

Technion Israel Institute of Technology, Technion City, Haifa, 3200003, Israel

\textit{Email address:} scherer@campus.technion.ac.il\ ,\ scherer@math.uni-sb.de

\end{document}